\documentclass[12pt]{amsart}
\usepackage[T1]{fontenc}
\usepackage[utf8]{inputenc}
\usepackage{color,amssymb,amsmath,mathrsfs,enumerate,esint}
\usepackage{a4wide}
\usepackage{hyperref}
\theoremstyle{plain}
\newtheorem{Theorem}{Theorem}[section]
\newtheorem{Lemma}[Theorem]{Lemma}

\theoremstyle{definition}

\newtheorem{Definition}[Theorem]{Definition}

\title[A note on the continuity of minors]
 {A note on the continuity of minors in \\
 grand Lebesgue spaces}

\author[A. Molchanova]{Anastasia Molchanova}

\address{%
Sobolev Institute of Mathematics \\
4 Acad. Koptyug avenue, Novosibirsk 630090, Russia\\
Peoples' Friendship University \\ 
6 Miklukho-Maklaya str., Moscow 117198, Russia}

\email{a.molchanova@math.nsc.ru}

\thanks{This work was supported by a
    Grant of the Russian Foundation of 
    the Russian Science Foundation
    (Agreement  No.~16-41-02004).}



\begin{document}

\maketitle

\begin{abstract}
We present a simple proof of the continuity,
in the sense distributions, 
of the minors of the differential matrices 
of mappings 
belonging to grand Sobolev spaces. 
Such function spaces were introduced in connection with a problem on minimal integrability of the Jacobian and are useful in certain aspects of geometric function theory and partial differential equations.
\end{abstract}

\section{Introduction}\label{sec:intro}

The academic literature on enlarged function spaces has grown considerably in recent times. 
Authors are typically concerned with the general theory of function spaces and applications in PDEs. 
An attractive feature of such function spaces is that they require a minimum of a priori assumptions, while member functions retain specific attractive properties such as continuity or regularity. 
Particular examples of these spaces are the so-called grand Lebesgue and grand Sobolev spaces. 
These spaces first appear in a paper by T.~Iwaniec and C.~Sbordone \cite{IwaSbo1992} in which they investigate minimal conditions for the integrability of the Jacobian of an orientation-preserving Sobolev mapping. 
Fundamental properties of these spaces have since been established such as duality and reflexivity \cite{Fio2000,FioKar2004}, as well as the boundedness of various integral operators \cite{JaiSinSin2016,Kok2010,KokMes2009}.
For further discussion of grand spaces, the interested reader is referred to \cite{CapFioKar2008,CasRaf2016,DonSboSch2013,FioGupJai2008,FioMerRak2001,FioMerRak2002,FioSbo1998,GreIwaSbo1997,JaiSinSin2017,Sbo1996,Sbo1998}. 

It is well known that if a sequence of mappings 
$f_m$ 
converges weakly in the Sobolev space
$W^{1,n}_{\rm loc}$
to a mapping 
$f_0$,
then 
all
$k\times k$-minors, 
$k=1, \dots n$, 
of matrices
$Df_m$
tend to the corresponding 
$k\times k$-minors
of the matrix
$Df_0$,
in the sense of distributions (\textit{in} $D'$),
see \cite[Ch. 9]{Mor1966} for the particular case
$n=2$, 
\cite[\S4.5]{Resh1982} and \cite[Theorem 8.20]{Dac2008} for 
$n\geq 2$.
The weak continuity of such minors plays a key role in the calculus of variations respecting the lower semicontinuity problem, 
see \cite{BenKru2017,Dac2008}
and references therein for more information. 
The related question of the integrability of the Jacobian (which is a particular case of a minor) under minimal assumptions, is partially motivated by applications such as nonlinear elasticity theory  \cite{Dac1982}.
Significant results were obtained for mappings with nonnegative Jacobians, which are sometimes called `orientation-preserving' mappings. 
Specifically, S.~M\"{u}ller proved that if 
$|Df| \in L^n$
and
$J_f(x)\geq 0$,
then the Jacobian possesses  
the higher integrability
$J_f(x) \in L \log L$
\cite{Mul1990}.
Further generalizations can be found in \cite{Gre1998,KosZho2002} and associated references.  
Following integrability, continuity theorems for Jacobians in corresponding spaces are the next natural step towards more general approximation results.
In this way T.~Iwaniec and A.~Verde obtained, in \cite{IwaVer1999}, the strong continuity of Jacobians in 
$L \log L$,
while L.~D'Onofrio and R.~Schiattarella in 
\cite{DonSch2013} proved a continuity theorem for 
orientation preserving mappings
$f_k$ 
belonging to the grand Sobolev space
$W^{1,n)}$.
Provided that 
we have the additional requirement of 
uniformly vanishing $n$-modulus, i.e.\
\begin{equation*}
    \lim\limits_{\varepsilon \to 0+} \varepsilon \sup\limits_{k \geq 1}
    \int\limits_{\Omega}|Df_k (x)|^{n-n\varepsilon} \,dx =0,
\end{equation*}
the weak continuity of Jacobians is obtained by  L.~Greco, T.~Iwaniec, and U.~Subramanian~\cite{GreIwaSub2003}. 

This paper proves continuity theorems for the minors of 
the differential matrix of mappings belonging to grand Sobolev spaces
(see Section~\ref{sec:preliminaries} for the definitions). 
More precisely,

\begin{Theorem}\label{th:main}
    Let
    $\Omega \subset \mathbb{R}^n$
    and  
    $f_m = (f_m^1, \dots, f_m^k)\colon \Omega \to \mathbb{R}^k$,
    $1 \leq k \leq n$,
    $m \in \mathbb{N}$,
    be a sequence of mappings locally bounded in 
    $W^{1,p),\delta} (\Omega)$
    with 
    $p > k$.
    Assume that 
    $f_m$
    converges in
    $L^{1}_{\rm loc}$ 
    to
    $f_0 = (f_0^1, \dots, f_0^k)$
    as 
    $m \to \infty$,
    then 
    the sequence of forms 
    $\omega_m = df_{m}^{1}\wedge \dots \wedge df_{m}^{k}$
    converges to
    $\omega_0 = df_{0}^{1}\wedge \dots \wedge df_{0}^{k}$
    in
    $D'$
    and is locally bounded in 
    $L^{\frac{p}{k}),\delta} (\Omega)$.
\end{Theorem}

The case 
$p=k$ 
requires some additional conditions, since it makes use of the property of the coincidence between the distributional Jacobian and the point-wise Jacobian
(Theorem~\ref{lem:weak_J} below). 
The same technique used in obtaining proof of the main result, with minor changes, allows us to prove the following results.

\begin{Theorem}\label{thm:main-2}
    Let
    $f_m = (f_{m}^{1}, \dots, f_{m}^{k})$,
    $1 \leq k \leq n$,
    $m \in \mathbb{N}$,
    be a sequence of mappings 
    locally bounded in 
    $W^{1,k)}$
    and
    with 
    $Df_m \in L^{k)}_{b}$.
    Assume that 
    $f_m$
    converges in
    $L^{1}_{\rm loc}$ 
    to
    $f_0 = (f_{0}^{1}, \dots, f_{0}^{k})$
    as 
    $m \to \infty$
    and
    forms 
    $\omega_m = df_{m}^{1}\wedge \dots \wedge df_{m}^{k}$
    and
    $\omega_0 = df_{0}^{1}\wedge \dots \wedge df_{0}^{k}$
    are locally integrable.
    It follows that 
    $\omega_m$
    converges to
    $\omega_0$
    in
    $D'$
    and is locally bounded in 
    $L^{1)}$.
\end{Theorem}

\begin{Theorem}\label{thm:main-3}
    Let
    $f_m = (f_{m}^{1}, \dots, f_{m}^{k})$,
    $1 \leq k \leq n$,
    $m \in \mathbb{N}$,
    be a sequence of mappings 
    locally bounded in 
    $W^{1,k)}$
    and
    with 
    $Df_m \in L^{k)}_{b}$.
    Assume that 
    $f_m$
    converges in
    $L^{1}_{\rm loc}$ 
    to
    $f_0 = (f_{0}^{1}, \dots, f_{0}^{k})$
    as 
    $m \to \infty$
    and
    all $k$-minors of matrix 
    $Df_m$
    are nonnegative.
    It follows that
    $\omega_m = df_{m}^{1}\wedge \dots \wedge df_{m}^{k}$
    converges to
    $\omega_0 = df_{0}^{1}\wedge \dots \wedge df_{0}^{k}$
    in
    $D'$
    and is locally bounded in 
    $L^{1)}$.
\end{Theorem}

The stated results are similar to those of \cite{GreIwaSub2003} but the proof, based on a technique used by Yu.~Reshetnyak \cite{Resh1982}, is comparatively simple and requires us to know only basic properties of the theory of differential forms and Sobolev spaces.
Moreover, this method allows us to easily extend the results for grand Sobolev spaces
$W^{1,p)}$ 
to grand Sobolev spaces with respect to measurable functions 
$W^{1,p),\delta}$,
as stated in Theorem~\ref{th:main}.

\newpage
\section{Preliminaries}\label{sec:preliminaries}

For a bounded open subset 
$\Omega$
in 
$\mathbb{R}^n$,
$n\geq 1$,
vector functions
$f = (f^1, \dots, f^n) \colon \Omega \to \mathbb{R}^n$
are called mappings of the Sobolev class
$W^{1,p} (\Omega,\mathbb{R}^n)$,
$1\leq p \leq \infty$, 
if all coordinate functions
$f^i$,
$i = 1,2,\dots, n$,
belong to 
$W^{1,p} (\Omega,\mathbb{R})$.
Throughout this paper the symbol 
$Df$ 
stands for the differential matrix 
and 
$J_f$ 
denotes its determinant, the Jacobian.

\begin{Definition}
    For 
    $0 < q < \infty$
    the \textit{grand Lebesgue space} 
    $L^{q)}(\Omega)$
    consists of all measurable functions
    $f \colon \Omega \to \mathbb{R}$
    such that
    \begin{equation}\label{def:grand_norm}
        \|f\|_{L^{q)}} = \sup\limits_{0 < \varepsilon < \varepsilon_0} 
        \left(\frac{\varepsilon}{|\Omega|} \int\limits_{\Omega} |f(x)|^{q-\varepsilon} \, dx \right)^{\frac{1}{q-\varepsilon}} < \infty,
    \end{equation}
    where 
    $\varepsilon_0 = q-1$ 
    if 
    $q > 1$
    and
    $\varepsilon_0 \in (0,q)$
    if
    $0<q\leq 1$.
\end{Definition}

Grand Lebesgue spaces have been thoroughly studied by many different authors. 
We refer the interested reader to the reviews given in articles \cite{DonSboSch2013,FioForGog2018,JaiSinSin2017} and \cite[\S 7.2]{CasRaf2016}.
However, we now state some basic properties of these spaces which will be useful for the results that follow.

For the case 
$q>1$, 
the continuous embeddings
$$L^{q} \subset L^{q)} \subset L^{q-\varepsilon}, \quad
\text{ for } 
0 < \varepsilon < q-1,$$
hold, and are strict. 
This can be easily seen by considering  
a unit ball
$B(0,1)$
and the function
$f(x) = |x|^{-\frac{n}{q}}$.
In this case 
$f$
belongs to 
$L^{q)}(B(0,1))$
but not 
$L^{q}(B(0,1))$.

Spaces 
$L^{q)}$
for 
$q>1$
are known to be non-reflexive Banach spaces 
\cite{Fio2000}.

\begin{Definition} 
    The space 
    $L_b^{q)}$
    consists of all functions 
    $f \in L^{q)}$
    such that
    \begin{equation*}
        \lim\limits_{\varepsilon \to 0+} \varepsilon \int\limits_{\Omega} |f(x)|^{q-\varepsilon} \, dx = 0.
    \end{equation*}
\end{Definition}

The space 
$L_b^{q)}$
is 
the closure of 
$L^q$ 
in the norm 
$\|\cdot\|_{L^{q)}}$
and 
$L_b^{q)} \neq L^{q)}$
see \cite{CarSbo1997,Gre1993}.
The validity of this latter claim is easy to see by considering  once again the function 
$f(x) = |x|^{-\frac{n}{q}}$
on the unit ball 
$B(0,1)$, 
for which $f\not\in L_b^{q)}(B(0,1))$,
since
$$\varepsilon \int\limits_{\Omega} |f(x)|^{q-\varepsilon} \, dx = \frac{q}{n} |B(0,1)| \not \to 0
\text{ as } 
\varepsilon \to 0+.$$

The embeddings 
$L^{q, p} \subset L^{q, \infty} \subset L^{q)}$
and
$L^{q}(\log L)^{-1} \subset L_b^{q)} \subset L^{q)}$
also hold, 
where 
$L^{q, p}$ are Lorenz spaces, 
and
$L^{q}(\log L)^{-1}$
are Orlicz spaces.
For further discussions of embeddings of these spaces, we refer the reader to \cite{DonSboSch2013,GiaGrePas2010,Gre1993,IwaSbo1992}.

As is seen from \eqref{def:grand_norm}, grand Lebesgue spaces can be characterized as controlling the blow-up of the Lebesgue norm by the parameter 
$\varepsilon$.
Indeed, the norm of the function
$f$,
belonging to
$\bigcap\limits_{0<\varepsilon < q-1} L^{q-\varepsilon}$
but not
$L^q$,
must blow up,
i.e.,
$\|f\|_{L^{q-\varepsilon}} \to \infty$,
when 
$\varepsilon \to 0$.
Thus, a natural generalization is to substitute for
$\varepsilon$
a measurable function 
$\delta(\varepsilon)$,
which is positive a.e.\ \cite{CapForGio2013}.

\begin{Definition}\label{def:GLd}
    For $0 < q < \infty$ 
    the \textit{grand Lebesgue space} 
    $L^{q), \delta}(\Omega)$
    \textit{with respect to 
    $\delta$}
    consists of all measurable functions
    $f \colon \Omega \to \mathbb{R}$
    such that
    \begin{equation*}
        \|f\|_{L^{q),\delta}} = \sup\limits_{0 < \varepsilon < \varepsilon_0} 
        \left(\frac{\delta(\varepsilon)}{|\Omega|} \int\limits_{\Omega} |f(x)|^{q-\varepsilon} \, dx \right)^{\frac{1}{q-\varepsilon}} < \infty,
    \end{equation*}
    where
    $\delta \in L^{\infty}((0,\varepsilon_0), (0,1])$
    is a left continuous function 
    such that
    $\lim\limits_{\varepsilon \to 0+} \delta(\varepsilon) = 0$
    and 
    $\delta^{\frac{1}{q-\varepsilon}}(\varepsilon)$
    is nondecreasing,
    $\varepsilon_0 = q-1$ 
    if 
    $q > 1$
    and
    $\varepsilon_0 \in (0,q)$
    if
    $0<q\leq 1$.
\end{Definition}
If 
$\delta(\varepsilon) = \varepsilon$,
the space 
$L^{q),\delta}$
is equivalent to 
$L^{q)}$.
If 
$\delta(\varepsilon) = \varepsilon^{\theta}$
with
$\theta >0$,
we denote the resulting space by 
$L^{q),\theta}$.
It was first introduced and studied in \cite{GreIwaSbo1997}. 
In \cite{CapForGio2013} it was also shown that for 
$q>1$
$$L^{q} \subset L^{q),\delta} \subset L^{q-\varepsilon} \quad
\text{ for } 
0 < \varepsilon \leq q-1.$$

The definition of convergence in the sense of distributions is standard.
    We say that the sequence 
    $f_m \in X(\Omega)$
    converges \textit{in the sense of distributions} 
    (\textit{in} 
    $D'$)
    to 
    $f_0$
    if, for every  function 
    $\varphi \in C_0^{\infty}(\Omega)$,
    \begin{equation*}
        \int\limits_\Omega f_m(x) \varphi(x) \, dx \to \int\limits_\Omega f_0(x) \varphi(x) \, dx \quad \text{as } m \to \infty.
    \end{equation*}
It is well-known that 
$f_m$
converges to 
$f_0$
weakly in
$L^{p}$
if and only if the sequence
$\{f_m\}_{m\in\mathbb{N}}$
is bounded in 
$L^{p}$
and 
$f_m$
converges in the sense of distributions
to 
$f_0$.


We now make some brief comments on exterior algebra that will be useful for the results that follow. 
Let 
$\omega$
be differential $k$-forms,
where
$1 \leq k \leq n$.
If 
$I = (i_1,i_2, \dots, i_k)$
is a $k$-tuple
with
$1 \leq i_1 < i_2 < \dots < i_k \leq n$,  
a differential form
$\omega$
can be represented as 
\begin{equation*}
    \omega = \sum \limits_{I} \omega_{I}(x)\, dx^{i_1} \wedge \dots \wedge dx^{i_k} = \sum \limits_{I} \omega_{I}(x)\, dx^{I}.
\end{equation*}
Note that the sequence of $k$-forms
$\omega_m$
converges 
to 
$\omega_0$
in  
$D'$
as 
$m \to \infty$
if the coefficients of the forms 
$\omega_m$
converge in  
$D'$
to the corresponding 
coefficients of 
$\omega_0$.

The calculus of differential forms is a powerful tool in the study of the analytical and geometrical properties of mappings. 
Thus, for mappings 
$f$
in Sobolev class 
$W^{1,p}$,
with 
$p\geq n$,
the Jacobian can be represented by the $n$-form
$$
    J_f =df_1\wedge \dots \wedge df_n. 
$$

To deal with the borderline case $p=k$
we need the integration-by-parts formula,
\begin{equation}\label{eq:Det=det}
    \int\limits_{\Omega} \varphi(x) J_f(x) \,dx = 
    - \int\limits_{\Omega} f^n \, df^1 \wedge df^2 \dots \wedge df^{n-1} \wedge d \varphi.
\end{equation}
It is easy to see that 
\eqref{eq:Det=det}
holds for 
$f\in W^{1,n}(\Omega)$.
In general, Sobolev embeddings and the H\"older inequality ensure that
for 
$f\in W^{1,\frac{n^2}{n+1}}_{\rm loc} (\Omega)$,
the right-hand-side of \eqref{eq:Det=det} can be considered as a distribution,
called 
    the \textit{distributional Jacobian}
    $\mathcal{J}_f$, 
    and
    defined 
    by the rule
    $$
        \mathcal{J}_f[\varphi] = - \int\limits_{\Omega} f^n \, df^1 \wedge df^2 \dots \wedge df^{n-1} \wedge d \varphi
    $$
    for every test function
    $\varphi \in C^{\infty}_0 (\Omega)$.
A function 
$f = x+ \frac{x}{|x|}$,
with 
$\Omega$
being a unit ball, 
shows that 
\eqref{eq:Det=det}
fails as soon as 
$f\in W^{1,p}(\Omega)$,
$p<n$.
The natural question of the coincidence of the distributional and the point-wise Jacobians is thoroughly studied in
\cite{Gre1993,IwaSbo1992,Mul1990}, as well as in \cite[\S 7.2]{IwaMar2001} and \cite[\S 6.2]{HajIwaMalOnn2008}.
We need the following results for grand Lebesgue spaces.

\begin{Lemma}[\hspace{-.5pt}{\cite[Theorem~4.1]{Gre1993}}]\label{lem:weak_J}
    Let 
    $f = (f^1, \dots, f^n) \in W^{1,1}_{\rm loc} (\Omega)$
    be a function such that 
    $J_f \in L^{1}_{\rm loc} (\Omega)$
    and 
    $|Df| \in L^{n)}_{b} (\Omega)$.
    Then \eqref{eq:Det=det} holds
    for all compactly supported test functions
    $\varphi \in C^{\infty}_0 (\Omega)$.
\end{Lemma}

\begin{Lemma}[\hspace{-.5pt}{\cite[Corollary~4.1]{Gre1993}}]\label{lem:weak_J>0}
    Let 
    $f = (f^1, \dots, f^n) \in W^{1,1}_{\rm loc} (\Omega)$
    be a function such that 
    $J_f (x) \geq 0$ 
    a.e.\ in 
    $\Omega$
    and 
    $|Df| \in L^{n)}_{b} (\Omega)$.
    Then \eqref{eq:Det=det} holds
    for all compactly supported test functions
    $\varphi \in C^{\infty}_0 (\Omega)$.
\end{Lemma}

Before we proceed to the proof of the main results, we need the following auxiliary lemma, which can be found in \cite[\S 4.5]{Resh1982},
and for which we now provide a proof for the convenience of the reader.

\begin{Lemma}\label{lem:conv_dif}
    Let 
    $\omega_m$
    be a sequence of differential $k$-forms, 
    bounded in 
    $L^{1}_{\rm loc} (\Omega)$,
    that converges in  
    $D'$
    to a form
    $\omega_0$
    as
    $m \to \infty$.
    Assume that each of the forms
    $\omega_m$,
    $m \in \mathbb{N}$,
    has in 
    $\Omega$
    a generalized differential,
    and that the sequence 
    $d\omega_m$
    is bounded in 
    $L^{1}_{\rm loc} (\Omega)$.
    It follows that the forms
    $d\omega_m$
    converge to
    $d\omega_0$
    in
    $D'$
    as 
    $m \to \infty$.
\end{Lemma}

\begin{proof}
    Consider an arbitrary
    $C^\infty$-smooth, compactly supported
    $(n-k-1)$-form
    $\alpha$.
    From the definition of a generalized differential we have
    $$
        \int\limits_\Omega \omega_m \wedge d\alpha = (-1)^{k-1}\int\limits_\Omega d\omega_m \wedge \alpha.
    $$
    Since 
    $\omega_m \to \omega_0$ in $D'$
    and 
    $d\alpha $
    is a
    $(n-k)$-form of the class $C_0^\infty (\Omega)$, 
    we obtain
    $$
        \int\limits_\Omega \omega_m \wedge d\alpha \xrightarrow[m \to \infty]{} \int\limits_\Omega \omega_0 \wedge d\alpha = (-1)^{k-1}\int\limits_\Omega d\omega_0 \wedge \alpha.
    $$
    And finally 
    $$
        \int\limits_\Omega d\omega_m \wedge \alpha \xrightarrow[m \to \infty]{}  \int\limits_\Omega d\omega_0 \wedge \alpha
    $$
    for all test $(n-k-1)$-forms
    $\alpha \in C_0^\infty (\Omega)$.
\end{proof}

We now make use of Lemma~\ref{lem:conv_dif} for grand Lebesgue spaces.

\begin{Lemma}\label{lem:conv_dif_GL}
    Let 
    $\omega_m$
    be a sequence of differential $k$-forms, 
    locally bounded in 
    $L^{p),\delta} (\Omega)$,
    that converges in  
    $D'$
    to a form
    $\omega_0$
    as
    $m \to \infty$.
    Assume that each of the forms
    $\omega_m$,
    $m \in \mathbb{N}$,
    has a generalized differential in 
    $\Omega$,
    and that the sequence 
    $d\omega_m$
    is locally bounded in 
    $L^{q),\delta} (\Omega)$.
    It follows that the forms
    $d\omega_m$
    converge to
    $d\omega_0$
    in
    $D'$
    as 
    $m \to \infty$.
\end{Lemma}

For a mapping 
$f = (f^1, \dots f^n)\colon \Omega \to \mathbb{R}^n$,
we define the 
$k \times k$-minors of the differential matrix as
\begin{equation*}
    \frac{\partial f^I}{\partial x^J} = \frac{\partial (f^{i_1}, \dots f^{i_k})}{\partial (x^{j_1}, \dots x^{j_k})}
\end{equation*}
for ordered $k$-tuples 
$I= (i_1,i_2, \dots, i_k)$
and 
$J= (j_1,j_2, \dots, j_k)$.
The representation 
\begin{equation*}
    df^{i_1} \wedge \dots \wedge df^{i_k} = \sum \limits_{J}\frac{\partial f^I}{\partial x^J}\, dx^{j_1} \wedge \dots \wedge dx^{j_k}
\end{equation*}
is valid.

Since in the proofs we investigate the properties of a particular 
$k \times k$ minor, 
it suffices to consider mappings 
$f \colon \Omega \to \mathbb{R}^k$
instead of maps into $\mathbb{R}^n$;
also, this makes the notation simpler.
Moreover, 
the condition
\textit{``$f_m$
converges in
$L^{1}_{\rm loc}$ 
to
$f_0$
as 
$m \to \infty$''}
results from the statement
\textit{``there exists a subsequence converging weakly in 
$W^{1,q}_{\rm loc}$
to 
$f_0$
for all 
$1\leq q < p$''}.
Indeed,
by the Sobolev embeddings we can find a subsequence 
$f_{m_l}$,
which converges to 
$f_0$
in 
$L^s_{\rm loc}$,
for some 
$1\leq s < \frac{nq}{n-q}$.
The H\"older inequality and boundedness of 
$\Omega$ 
then guarantee that 
$f_0$
is also an
$L^1_{\rm loc}$-limit of 
$f_{m_l}$.


\section{Proof of the main results}

    We will prove Theorem~\ref{th:main} by induction on 
    $k$. 
    The case of
    $k=1$ 
    follows directly from Lemma~\ref{lem:conv_dif_GL}.
    Assume that the lemma has been proven for some general 
    $k$,
    and let 
    $f_m\colon \Omega \to \mathbb{R}^{k+1}$
    be a sequence of mappings of class 
    $W^{1,p),\delta} (\Omega)$,
    $p > k+1$.
    The sequence 
    $f_m$ 
    is locally bounded in 
    $W^{1,p),\delta} (\Omega)$,
    consequently,
    also bounded in
    $W^{1,p-\varepsilon} (\Omega)$
    for
    $0 < \varepsilon < p-1$,
    and is locally convergent in 
    $L^1$
    to 
    $f_0$.
    From the Sobolev embedding theorem
    we obtain that
    $f_m \to f_0$
    in
    $L^s$ 
    for
    $s < \frac{n(p-\varepsilon)}{n-p + \varepsilon}$.
    \smallskip\\
    \textsc{Step I.} Let us consider the forms
    \begin{equation}\label{def:uvw}
        \begin{aligned}
            & u = dy^1 \wedge d y^2 \wedge \dots \wedge dy^k, \\
            & v = (-1)^k y^{k+1} u = (-1)^k y^{k+1} dy^1 \wedge \dots \wedge dy^k, \\
            & w = u \wedge d y^{k+1} = dy^1 \wedge d y^2 \wedge \dots \wedge dy^k \wedge d y^{k+1}
        \end{aligned}
    \end{equation}
    in 
    $\mathbb{R}^{k+1}$.

    It is easy to see that 
    $w = d v$.

    Consider also the pull-backed forms
    \begin{equation}\label{def:pullbacked_forms}
        \begin{aligned}
            & \tilde \omega_m = f^*_m u = df_{m}^{1} \wedge d f_{m}^{2} \wedge \dots \wedge df_{m}^{k}, \\
            & \psi_m = f^*_m v = (-1)^k f_{m}^{k+1} \tilde \omega_m, \\
            & \omega_m = f^*_m w = df_{m}^{1} \wedge d f_{m}^{2} \wedge \dots \wedge df_{m}^{k+1}.
        \end{aligned}
    \end{equation}
    Then 
    $\omega_m = d \psi_m$
    for each 
    $m$.
    In fact 
    $\omega_m$,
    $\psi_m \in L^1$,
    since 
    for each of
    $j$,
    the functions
    $f_{m}^{j}$,
    $df_{m}^{j}$ lie in
    $L^{\tilde{p}}$,
    where
    $p > \tilde{p} \geq k+1$.
    Thus,
    for any $(n-k-1)$-form 
    $\eta \in C_0^\infty (\Omega)$,
    \begin{equation} \label{eq:weak_J}
        \int_\Omega \omega_m \wedge \eta = (-1)^{k-1} \int_\Omega \psi_m \wedge d \eta.
    \end{equation}
    By the induction hypothesis
    $\tilde \omega_m \to \tilde \omega_0$
    in 
    $D'$
    and 
    $\tilde \omega_m$
    is locally bounded in 
    $L^{p/k)}$.
    \smallskip\\
    \textsc{Step II.} 
    Let 
    $\xi$
    be an arbitrary $C^\infty$-smooth, compactly supported 
    $(n-k)$-form.
    Let us show that
    \begin{equation}\label{eq:conv_0}
        \int_\Omega f_{m}^{k+1} \tilde \omega_m \wedge \xi \to 
        \int_\Omega f_{0}^{k+1} \tilde \omega_0 \wedge \xi.
    \end{equation}
    Indeed, fix $0 < \varepsilon = \frac{k}{n+1} < p-1$. 
    Then, by the Sobolev embedding theorem,
    $f_{m}^{k+1} \to f_{0}^{k+1}$
    in
    $L^s$,
    $s < \frac{n(p-\varepsilon)}{n-p + \varepsilon}$.
    Put 
    $s' = \frac {p- \varepsilon}{k}$
    and
    $s = \frac{p - \varepsilon }{p - k - \varepsilon}$,
    then 
    $\frac{1}{s'} + \frac{1}{s} = 1$.
    Hence
    \begin{equation}\label{eq:conv_1}
        \left| \int_\Omega f_{m}^{k+1} \tilde \omega_m \wedge \xi - 
        \int_\Omega f_{0}^{k+1} \tilde \omega_m \wedge \xi \right|   \\
        \leq C \|\tilde \omega_m\|_{L^{s'} (A)} \|f_{m}^{k+1} - f_{0}^{k+1}\|_{L^{s} (A)} 
        \to 0,
    \end{equation}
    where 
    $A = \operatorname{supp} \xi$.
    Further, for any 
    $\gamma >0$ 
    let 
    $f \in C_0^\infty (\Omega)$ 
    be such that 
    $\|f - f_{0}^{k+1}\|_{L^s(\Omega)} < \gamma$.
    Then
    \begin{multline*} 
        \left| \int_\Omega f_{0}^{k+1} \tilde \omega_m \wedge \xi - 
        \int_\Omega f_{0}^{k+1} \tilde \omega_0 \wedge \xi \right| \leq 
        \left| \int_\Omega (f_{0}^{k+1} - f) \tilde \omega_m \wedge \xi \right| \\
        + 
        \left| \int_\Omega f (\tilde \omega_m \wedge \xi - \tilde \omega_0 \wedge \xi) \right|
        + \left| \int_\Omega (f - f_{0}^{k+1}) \tilde \omega_0 \wedge \xi \right| \to 0
    \end{multline*}
    as 
    $m \to \infty$.
    The first and the third terms are less than 
    $C\gamma$
    due to the choice of 
    $f$,
    the second one tends to zero by the induction hypothesis.
    Since 
    $\gamma$
    is arbitrary,
    this implies that 
    \begin{equation}\label{eq:conv_2}
        \int_\Omega f_{0}^{k+1} \tilde \omega_m \wedge \xi \to 
        \int_\Omega f_{0}^{k+1} \tilde \omega_0 \wedge \xi.
    \end{equation}
    The convergence \eqref{eq:conv_0} follows from \eqref{eq:conv_1} and \eqref{eq:conv_2}.
    This means that the sequence of forms 
    $\psi_m = f_{m}^{k+1} \tilde \omega_m$
    converges to the form 
    $\psi_0 = f_{0}^{k+1} \tilde \omega_m$
    in 
    $D'$.
    
    It remains to show that 
    the sequences of forms 
    $\psi_m$ 
    and
    $d \psi_m$
    are bounded in 
    $L^{q),\delta}$
    for 
    $q = \frac{p}{k+1}$.
    Indeed, the H\"older inequality provides

    \begin{multline*}
        \bigg(\int_\Omega |\psi_m|^{q - \varepsilon} \, dx \bigg)^{\frac{1}{q-\varepsilon}} = 
        \bigg(\int_\Omega |f_{m}^{k+1} \tilde \omega_m|^{q - \varepsilon} \, dx \bigg)^{\frac{1}{q-\varepsilon}}  \\
        \leq
        \bigg(\int_\Omega |f_{m}^{k+1}|^{(q - \varepsilon) \frac{p - \varepsilon}{q - \varepsilon}} \, dx \bigg)^{\frac{1}{p-\varepsilon}} 
        \bigg(\int_\Omega |\tilde \omega_m|^{(q - \varepsilon) \frac{p - \varepsilon}{p - q}} \, dx \bigg)^{\frac{p-q}{(p-\varepsilon)(q - \varepsilon)}}.
    \end{multline*}
    Here 
    $\frac{p - \varepsilon}{q-\varepsilon} > 1$
    as 
    $p - \varepsilon > q - \varepsilon$.
    
    Multiplying by 
    $\delta(\varepsilon)$
    and taking the supremum, we obtain
    \begin{multline} \label{est:psi}
        \|\psi_m\|_{L^{q),\delta}}  \\
        \leq
        \sup\limits_{0 < \varepsilon < q-1}\bigg(\delta(\varepsilon)\int_\Omega |f_{m}^{k+1}|^{(q - \varepsilon) \frac{p - \varepsilon}{q - \varepsilon}} \, dx \bigg)^{\frac{1}{p-\varepsilon}} 
        \bigg(\delta(\varepsilon)\int_\Omega |\tilde \omega_m|^{(q - \varepsilon) \frac{p - \varepsilon}{p - q}} \, dx \bigg)^{\frac{p-q}{(p-\varepsilon)(q - \varepsilon)}}  \\
        \leq
        \sup\limits_{0 < \varepsilon < p-1}\bigg(\delta(\varepsilon)\int_\Omega |f_{m}^{k+1}|^{p - \varepsilon} \, dx \bigg)^{\frac{1}{p-\varepsilon}} 
        \sup\limits_{0 < \varepsilon' < \frac{p}{k}-1}\bigg(\delta(\varepsilon')\int_\Omega |\tilde \omega_m|^{\frac{p}{k} - \varepsilon'} \, dx \bigg)^{\frac{1}{p/k-\varepsilon'}} 
        \\
        \leq
        \|f_{m}^{k+1}\|_{L^{p),\delta}} \|\tilde \omega_m\|_{L^{p/k),\delta}}.
    \end{multline}

    The last inequality is valid for 
    $\varepsilon' = \frac{\varepsilon (2p + pk - \varepsilon k - \varepsilon)}{pk}$,
    which satisfies \\
    $\frac{(q - \varepsilon)(p - \varepsilon)}{p - q} = \frac{p}{k} - \varepsilon'$.
    It is easy to check that
    $\varepsilon < \varepsilon'$,
    and from Definition~\ref{def:GLd} we can deduce that
    $\delta$ is a nondecreasing function, 
    and thus 
    $\delta (\varepsilon)^{\frac{1}{p/k - \varepsilon'}} \leq \delta (\varepsilon')^{\frac{1}{p/k - \varepsilon'}}$. 

    In order to make sure that 
    $0 < \varepsilon' < \frac{p}{k}-1$,
    we show that
    \begin{equation*}
        h(\varepsilon) = pk \left(\frac{p}{k} -1 - \varepsilon'\right) = (k + 1) \varepsilon^2 - (2p + pk)\varepsilon + p^2 - pk >0.
    \end{equation*}
    First, note that
    $h(0) > 0$
    and 
    $h\left(\frac{p}{k+1} - 1\right) >0$.
    Moreover, 
    $h'(\varepsilon) = 2 (k + 1) \varepsilon - (2p + pk) < 0$
    if
    $\varepsilon < \frac{2p + pk}{2 (k + 1)}$
    with 
    $\frac{p}{k+1} - 1 < \frac{2p + pk}{2 (k + 1)}$,
    i.e., 
    $h(\varepsilon)$ decreases for 
    $0< \varepsilon < \frac{p}{k+1} - 1$
    and takes positive values at the boundary points.
    Thus, 
    $h(\varepsilon)>0$
    for all 
    $\varepsilon \in (0, \frac{p}{k+1} - 1)$,
    and so it follows that
    $0 < \varepsilon' < \frac{p}{k}-1$.

    In view of this, we can consider the supremum over all 
    $0 < \varepsilon' < \frac{p}{k} - 1$,
    and its value 
    is not less than the supremum over all 
    $0 < \varepsilon < q - 1=\frac{p}{k+1} - 1$.
    This completes the proof of \eqref{est:psi}.

    The same arguments show that 
    $d \psi_m = \omega_m = \tilde \omega_m \wedge d f_{m}^{k+1}$
    is bounded in 
    $L^{\frac{p}{k+1}),\delta}$.
    By Lemma~\ref{lem:conv_dif_GL}, this implies that 
    $\omega_m \to \omega_0$
    in 
    $D'$.

\begin{proof}[Proof of Theorem~\ref{thm:main-2} and Theorem~\ref{thm:main-3}]
    Here, we need some modifications of the proof of Theorem~\ref{th:main}. At Step I we use Lemma~\ref{lem:weak_J} to obtain the relation~\eqref{eq:weak_J}. Note that Lemma~\ref{lem:weak_J} can be modified for $k$-forms
    by considering 
    $f^I = (f^1,f^2, \dots, f^k, x^{i_{k+1}}, \dots, x^{i_n})$,
    where 
    $x^{i_l}$
    is a corresponding coordinate function.
    \smallskip\\
    \textsc{Step I.}
    Recall that 
    $p=k+1$. 
    Let us consider the forms
    $u$, $v$, $w$ and their pullbacks
    $\tilde \omega_m$, $\psi_m$, and $\omega_m$
    defined by \eqref{def:uvw} and \eqref{def:pullbacked_forms},
    correspondingly.
    Now we use Lemma~\ref{lem:weak_J} to obtain
    $\omega_m = d \psi_m$
    for each 
    $m$.
    
    Indeed, 
    $\omega_m = df_{m}^{1} \wedge d f_{m}^{2} \wedge \dots \wedge df_{m}^{k+1} \in L^1_{\rm loc}$ 
    by the hypothesis of Theorem~\ref{thm:main-2},
    the local integrability of
    $\psi_m = (-1)^k f_{m}^{k+1} df_{m}^{1} \wedge d f_{m}^{2} \wedge \dots \wedge df_{m}^{k}$
    follows from
    $f\in W_{\rm loc}^{1,\frac{n(k+1)}{n+1}}$,
    as $\frac{n(k+1)}{n+1} < k+1$.
    Then we have
    $df_{m}^{1} \wedge d f_{m}^{2} \wedge \dots \wedge df_{m}^{k} \in L_{\rm loc}^{\frac{n(k+1)}{k(n+1)}}$
    and, from the Sobolev embedding theorem
    $f_{m}^{k+1} \in L_{\rm loc}^{\frac{n(k+1)}{n-k}}$.
    The H\"older inequality provides the required integrability,
    as
    $\frac{k(n+1)}{n(k+1)} + \frac{n-k}{n(k+1)} = 1$.

    Hence,
    for any $(n-k-1)$-form 
    $\eta \in C_0^\infty (\Omega)$,
    \begin{equation*}
        \int_\Omega \omega_m \wedge \eta = (-1)^{k-1} \int_\Omega \psi_m \wedge d \eta.
    \end{equation*}
    By the induction hypothesis
    $\tilde \omega_m \to \tilde \omega_0$
    in 
    $D'$
    and the sequence
    $\tilde \omega_m$
    is
    locally bounded in 
    $L^{p/k)}$.
    \smallskip\\
    \textsc{Step II.} 
    All the estimates of Step II in the proof of Theorem~\ref{th:main} are satisfied if we consider in the definition 
    of the grand Lebesgue norm 
    $\varepsilon_0 = \frac{k+2 - \sqrt{k^2+4k}}{2} < 1$.
    According to Lemmas~\ref{lem:weak_J} and \ref{lem:weak_J>0}, we can replace the local integrability condition of
    $\omega_m$
    by non-negativity of all $k$-minors of the matrix 
    $Df_m$.

    Let 
    $\xi$
    be an arbitrary $C^\infty$-smooth, compactly supported 
    $(n-k)$-form.
    Let us show that
    \begin{equation}\label{eq:conv_0_1}
        \int_\Omega f_{m}^{k+1} \tilde \omega_m \wedge \xi \to 
        \int_\Omega f_{0}^{k+1} \tilde \omega_0 \wedge \xi.
    \end{equation}
    To this end, fix $0 < \varepsilon = \frac{k}{n+1} < k = p-1$. 
    From the Sobolev embedding theorem
    $f_{m}^{k+1} \to f_{0}^{k+1}$
    in
    $L^s$,
    $s < \frac{n(k+1-\varepsilon)}{n-k-1 + \varepsilon}$.
    Put 
    $s' = \frac {k+1- \varepsilon}{k}$
    and
    $s = \frac{k+1 - \varepsilon }{1 - \varepsilon}$,
    then 
    $\frac{1}{s'} + \frac{1}{s} = 1$.
    Hence
    \begin{equation}\label{eq:conv_1_1}
        \left| \int_\Omega f_{m}^{k+1} \tilde \omega_m \wedge \xi - 
        \int_\Omega f_{0}^{k+1} \tilde \omega_m \wedge \xi \right|  \\
        \leq C \|\tilde \omega_m\|_{L^{s'} (A)} \|f_{m}^{k+1} - f_{0}^{k+1}\|_{L^{s} (A)} 
        \to 0,
    \end{equation}
    where 
    $A = \operatorname{supp} \xi$.
    Furthermore,
    for any
    $\gamma > 0$ 
    let 
    $f \in C_0^\infty (\Omega)$ 
    be such that 
    $\|f - f_{0}^{k+1}\|_{L^s(\Omega)} < \gamma$,
    then
    \begin{multline*} 
        \left| \int_\Omega f_{0}^{k+1} \tilde \omega_m \wedge \xi - 
        \int_\Omega f_{0}^{k+1} \tilde \omega_0 \wedge \xi \right| \leq 
        \left| \int_\Omega (f_{0}^{k+1} - f) \tilde \omega_m \wedge \xi \right| \\
         +
        \left| \int_\Omega f (\tilde \omega_m \wedge \xi - \tilde \omega_0 \wedge \xi) \right|
        + \left| \int_\Omega (f - f_{0}^{k+1}) \tilde \omega_0 \wedge \xi \right| \to 0
    \end{multline*}
    as 
    $m \to \infty$.
    The first and the third terms are less than 
    $C\gamma$
    due to the choice of 
    $f$,
    and the second one tends to zero by the induction hypothesis.
    Since 
    $\gamma$
    is arbitrary,
    this implies that 
    \begin{equation}\label{eq:conv_2_1}
        \int_\Omega f_{0}^{k+1} \tilde \omega_m \wedge \xi \to 
        \int_\Omega f_{0}^{k+1} \tilde \omega_0 \wedge \xi.
    \end{equation}
    The relation indicated in \eqref{eq:conv_0_1} follows from \eqref{eq:conv_1_1} and \eqref{eq:conv_2_1}.
    This means that the sequence of forms 
    $\psi_m = f_{m}^{k+1} \tilde \omega_m$
    converges to the form 
    $\psi_0 = f_{0}^{k+1} \tilde \omega_m$
    in 
    $D'$.
    
    It remains to check that 
    the sequences of forms 
    $\psi_m$ 
    and
    $d \psi_m$
    are bounded in 
    $L^{1)}$.
    The H\"older inequality provides
    \begin{multline*}
        \bigg(\int_\Omega |\psi_m|^{1 - \varepsilon} \, dx \bigg)^{\frac{1}{1-\varepsilon}} = 
        \bigg(\int_\Omega |f_{m}^{k+1} \tilde \omega_m|^{1 - \varepsilon} \, dx \bigg)^{\frac{1}{1-\varepsilon}}  \\
        \leq
        \bigg(\int_\Omega |f_{m}^{k+1}|^{(1 - \varepsilon) \frac{k+1 - \varepsilon}{1- \varepsilon}} \, dx \bigg)^{\frac{1}{k+1-\varepsilon}} 
        \bigg(\int_\Omega |\tilde \omega_m|^{(1 - \varepsilon) \frac{k+1 - \varepsilon}{k}} \, dx \bigg)^{\frac{k}{(k+1-\varepsilon)(1 - \varepsilon)}};
    \end{multline*}
    here 
    $\frac{k+1 - \varepsilon}{1-\varepsilon} > 1$.
    
    Multiplying by 
    $\varepsilon$
    and taking the supremum, we obtain
    \begin{multline} \label{est:psi_1}
        \|\psi_m\|_{L^{1)}}  \\
        \leq
        \sup\limits_{0 < \varepsilon < \varepsilon_0}\bigg(\varepsilon\int_\Omega |f_{m}^{k+1}|^{(1 - \varepsilon) \frac{k+1 - \varepsilon}{1 - \varepsilon}} \, dx \bigg)^{\frac{1}{k+1-\varepsilon}} 
        \bigg(\varepsilon\int_\Omega |\tilde \omega_m|^{(1 - \varepsilon) \frac{k+1 - \varepsilon}{k}} \, dx \bigg)^{\frac{k}{(k+1-\varepsilon)(1 - \varepsilon)}}  \\
        \leq
        \sup\limits_{0 < \varepsilon < k}\bigg(\varepsilon\int_\Omega |f_{m}^{k+1}|^{k+1 - \varepsilon} \, dx \bigg)^{\frac{1}{k+1-\varepsilon}} 
        \sup\limits_{0 < \varepsilon' < \frac{k+1}{k}-1}\bigg(\varepsilon'\int_\Omega |\tilde \omega_m|^{\frac{k+1}{k} - \varepsilon'} \, dx \bigg)^{\frac{1}{(k+1)/k-\varepsilon'}} 
        \\
        \leq
        \|f_{m}^{k+1}\|_{L^{k+1)}} \|\tilde \omega_m\|_{L^{\frac{k+1}{k})}}.
    \end{multline}

    The last inequality is valid for 
    $\varepsilon' = \frac{\varepsilon (2 + k - \varepsilon )}{k}$,
    which satisfies
    $\frac{(1 - \varepsilon)(k+1 - \varepsilon)}{k} = \frac{k+1}{k} - \varepsilon'$.
    It is easy to check that
    $\varepsilon < \varepsilon'$.
    In order to make sure that 
    $0 < \varepsilon' < \frac{k+1}{k}-1 = \frac{1}{k}$,
    note that the roots of
    $h(\varepsilon) = \varepsilon^2 - (2 + k)\varepsilon + 1 $,
    $\varepsilon_{1,2} = \frac{k+2 \pm \sqrt{k^2+4k}}{2}$
    are not less than 
    $\varepsilon_0 = \frac{k+2 - \sqrt{k^2+4k}}{2}$,
    and 
    $h(0) = 1 > 0$.

    In view of this, we can consider the supremum over all 
    $0 < \varepsilon' < \frac{1}{k}$,
    and, by doing so, its value is seen to increase.
    This completes the proof of the estimate \eqref{est:psi_1}.

    The same arguments show that 
    $d \psi_m = \omega_m = \tilde \omega_m \wedge d f_{m}^{k+1}$
    is bounded in 
    $L^{\frac{p}{k+1})}$.
    By Lemma~\ref{lem:conv_dif_GL}, this implies that 
    $\omega_m \to \omega_0$
    in 
    $D'$.
\end{proof}

\subsection*{Acknowledgment} 
The author warmly thanks professor Sergey Vodopyanov and my great friend Dr.\ Ian McGregor for the numerous discussions on, and useful comments about this paper.


\begin{thebibliography}{100}

\bibitem{BenKru2017} B. Bene\v{s}ov{\'a} and M. Kru\v{z}{\'i}k, 
\textit{Weak lower semicontinuity of integral functionals and applications.}
SIAM Rev. \textbf{59}(4) (2017), 703--766.

\bibitem{CapFioKar2008} C. Capone, A. Fiorenza, and G.E. Karadzhov, 
\textit{Grand {O}rlicz spaces and global integrability of the {J}acobian.}
Math. Scand. \textbf{102} (2008), 131--148.

\bibitem{CapForGio2013} C. Capone, M.R. Formica, and R. Giova, 
\textit{Grand Lebesgue spaces with respect to measurable
functions.}
Nonlinear Anal. \textbf{85} (2013), 125--131.

\bibitem{CapFio2005} C. Capone, A. Fiorenza, 
\textit{On small Lebesgue spaces.}
J. Funct. Spaces Appl. \textbf{3} (2005), 73--89.

\bibitem{CarSbo1997} M. Carozza and C. Sbordone,
\textit{The distance to {$L^{\infty}$} in some function spaces and applications.}
Differ. Integral Equ. Appl. \textbf{10} (1997), 599--607.

\bibitem{CasRaf2016} R.E. Castillo, H. Rafeiro
\textit{An Introductory Course in Lebesgue Spaces.}
Springer, Switzerland, 2016.

\bibitem{Dac1982} B. Dacorogna,
\textit{Weak continuity and weak lower semicontinuity of nonlinear functionals.}
Lect. Notes Math., Vol. 922, Springer, Berlin, 1982.

\bibitem{Dac2008} B. Dacorogna,
\textit{Direct Methods in the Calculus of Variations.}
2nd Edition, Springer, New York, 2008.

\bibitem{DonSboSch2013} L. D'Onofrio, C. Sbordone, and R. Schiattarella, 
\textit{{G}rand {S}obolev spaces and their applications in geometric function theory and {PDE}s.}
J. Fixed Point Theory Appl. \textbf{13} (2013), 309--340.

\bibitem{DonSch2013} L. D'Onofrio and R. Schiattarella,
\textit{On the continuity of {J}acobian of orientation preserving mappings in the grand {S}obolev space.}
Differ. Integral Equ. \textbf{9/10}(26) (2013), 1139--1148.

\bibitem{Fio2000} A. Fiorenza,
\textit{Duality and reflexivity in grand {L}ebesgue spaces.}
Collect. Math. \textbf{51} (2000), 131--148.

\bibitem{FioForGog2018} A. Fiorenza, M.R. Formica, A. Gogatishvili, 
\textit{On grand and small Lebesgue and Sobolev spaces and some applications to PDE.}
Differ. Equ. Appl. \textbf{10}(1) (2018), 21--46.

\bibitem{FioGupJai2008} A. Fiorenza, B. Gupta, and P. Jain,
\textit{The maximal theorem for weighted grand {L}ebesgue spaces.}
Studia Math. \textbf{188}(2) (2008), 123--133.

\bibitem{FioKar2004} A. Fiorenza and G.E. Karadzhov,
\textit{Grand and small {L}ebesgue spaces and their analogs.}
J. Anal. Appl. \textbf{23} (2004), 657--681.

\bibitem{FioMerRak2001} A. Fiorenza, A. Mercaldo, and J.M. Rakotoson,
\textit{Regularity and comparison results in grand {S}obolev spaces for parabolic equations with measure data.}
Appl. Math. Lett. \textbf{14} (2001), 979--981.

\bibitem{FioMerRak2002} A. Fiorenza, A. Mercaldo, and J.M. Rakotoson,
\textit{Regularity and uniqueness results in grand {S}obolev spaces for parabolic equations with measure data.}
Discrete Contin. Dyn. Syst. \textbf{8}(4) (2002), 893--906.

\bibitem{FioSbo1998} A. Fiorenza and C. Sbordone,
\textit{Existence and uniqueness results for solutions of nonlinear equations with right hand side in {$L^1$}.}
Studia Math. \textbf{127}(3) (1998), 223--231.

\bibitem{GiaGrePas2010} F. Giannetti, L. Greco, and A. Passarelli di Napoli, \textit{The self-improving property of the Jacobian determinant in Orlicz spaces.} 
Indiana Univ. Math. J. \textbf{59} (2010), 91--114.

\bibitem{Gre1993} L. Greco, 
\textit{A remark on the equality {$\det Df = \operatorname{Det} Df$}.}
Differ. Integral Equ. \textbf{6} (1993), 1089--1100.

\bibitem{Gre1998} L. Greco,
\textit{Sharp integrability of nonnegative {J}acobians.}
Rend. Mat. Appl. \textbf{18} (1998), 585--600.

\bibitem{GreIwaSbo1997} L. Greco, T. Iwaniec, and C. Sbordone,
\textit{Inverting the {$p$}-harmonic operator.}
Manuscripta Math. \textbf{92} (1997), 249--258.

\bibitem{GreIwaSub2003} L. Greco, T. Iwaniec, and U. Subramanian,
\textit{Another approach to biting convergence of {J}acobians.}
Illinois J. Math. \textbf{24}(3) (2003), 815--830.

\bibitem{IwaMar2001} T. Iwaniec, and G.\,J. Martin.
\textit{Geometric Function Theory and Non-linear Analysis.}
Oxford Mathematical Monographs, 2001.

\bibitem{IwaSbo1992} T. Iwaniec and C. Sbordone,
\textit{On the integrability of the Jacobian under minimal
hypotheses.}
Arch. Ration. Mech. Anal. \textbf{119}(1992), 129--143.

\bibitem{IwaVer1999} T. Iwaniec and A. Verde,
\textit{On the operator {$L(f) = f \log |f|$}.}
J. Funct. Anal. \textbf{169} (1999), 391--420.

\bibitem{JaiSinSin2016} P. Jain, M. Singh, and A.P. Singh,
\textit{Hardy type operators on grand {L}ebesgue spaces for non-increasing functions.}
Trans. Razmadze Math. Inst. \textbf{170} (2016), 34--46.

\bibitem{JaiSinSin2017} P. Jain, M. Singh, and A.P. Singh,
\textit{Recent Trends in Grand Lebesgue Spaces.} 
In: P. Jain, H-J. Schmeisser (eds.) Function Spaces and Inequalities. Springer Proceedings in Mathematics \& Statistics, Vol.~\textbf{206}. Springer, Singapore, 2017.

\bibitem{HajIwaMalOnn2008} P. Haj\l asz, T. Iwaniec, J. Mal\'y, J. Onninen,
\textit{Weakly Differentiable Mappings Between Manifolds.} 
Mem. Amer. Math. Soc., \textbf{192}:899, 2008.

\bibitem{Kok2010} V. Kokilashvili,
\textit{Boundedness criterion for singular integrals in weighted grand {L}ebesgue spaces.}
J. Math. Sci. \textbf{170} (2010), 20--33.

\bibitem{KokMes2009} V. Kokilashvili and A. Meskhi,
\textit{A note on the boundedness of the {H}ilbert transform in weighted grand {L}ebesgue spaces.}
Georgian Math. J. \textbf{16} (2009), 547--551.

\bibitem{KosZho2002} P. Koskela and X. Zhong,
\textit{Minimal assumptions for the integrability of the {J}acobian.}
Ricerche Mat. \textbf{51}(2) (2002) 297--311.

\bibitem{Mul1990} S. M\"{u}ller,
\textit{Higher integrability of determinants and weak convergence in {$L_1$}.}
J. Reine Angew. Math. \textbf{412} (1990), 20--34.

\bibitem{Mor1966} C.~B. Morrey, 
\textit{Multiple integrals in the calculus of variations.} 
Springer, New York, 1966.

\bibitem{Resh1982}  Yu.~G. Reshetnyak,
\textit{Space mappings with bounded distortion.}
Transl. Math. Monographs 73, AMS, New York, 1989.

\bibitem{Sbo1996} C. Sbordone,
\textit{Grand {S}obolev spaces and their applications to variational problems.}
Matematiche \textbf{52}(2) (1996), 335-347.

\bibitem{Sbo1998} C. Sbordone,
\textit{Nonlinear elliptic equations with right hand side in nonstandard spaces},
Atti Sem. Mat. Fis. Univ. Modena, \textbf{46}(Suppl.) (1998) 361--368.

\end{thebibliography}
\end{document}